\documentclass[a4paper,10pt,leqno]{amsart}
\usepackage{amsmath,amsfonts,amssymb,amsthm,epsfig,epstopdf,url,array}
\usepackage{amsthm}
\usepackage{mathrsfs}
\usepackage{epsfig}
\usepackage{graphicx}

\newtheorem{thm}{Theorem}[section]
\newtheorem{prop}[thm]{Proposition}
\newtheorem{coro}[thm]{Corollary}
\newtheorem{lem}[thm]{Lemma}
\newtheorem{rmk}[thm]{Remark}

\numberwithin{equation}{section}

\newcommand{\N}{{\mathbb{N}}}
\newcommand{\z}{{\mathbb{Z}}}

\newcommand{\ceil}[1]{\left\lceil #1 \right\rceil}
\newcommand{\eq}{\equiv}

\title[Completing an universal $m$-gonal form]{ Completing an universal $m$-gonal form with a closing unary piece }

\author{Dayoon Park}
\address{Department of Mathematics, The University of Hong Kong, Hong Kong}
\email{pdy1016@hku.hk}
\thanks{}

\begin{document}
\maketitle

\begin{abstract}
In this paper, we show that for any $m$-gonal form $F_m(\mathbf x)$ with $m \ge 12$ which represents every positive integer up to $m-4$, by putting together only unary $m$-gonal form, we may complete an universal form.
\end{abstract}

\maketitle
\section{Introduction}
The representation of positive integers by {\it $m$-gonal numbers} which were defined as the number of dots to constitute a regular $m$-gon in BC 2-th Century has been studied for a long time.
The famous Fermat's conjecture which states that every positive integer may be written as a sum of at most $m$ $m$-gonal numbers date back to 17-th Century.
His conjecture was resolved by Lagrange and Gauss for $m=4$ and $m=3$ in 1770 and 1796, respectively.
And finally Cauchy completed its proof for general $m \ge 3$ in 1813.

One may easily derive that the number of dots to constitute a regular $m$-gon with $x$ dots for each side has a formula
\begin{equation}\label{m-number}P_m(x)=\frac{m-2}{2}x^2-\frac{m-4}{2}x. \end{equation}
We particularly call the $m$-gonal number $P_m(x)$ in (\ref{m-number}) as the {\it $x$-th $m$-gonal number}.
Although the original definition for $m$-gonal number admits only positive integer for variable $x$ in (\ref{m-number}),
one may naturally generalize the $m$-gonal number by admitting the $x$ in (\ref{m-number}) zero and negative integer too.
As a general version of Fermat's Conjecture, the author and collaborators \cite{KP} determined the optimal $\ell_m$ satisfying every positive integer is written as a sum of $\ell_m$ generalized $m$-gonal numbers for every $m \ge 3$ except $7$ and $9$ as
$$ \ell_m=
\begin{cases} m-4 & \text{ if } m \ge 10\\
3& \text{ if } m \in \{3,5,6\}\\
4& \text{ if } m \in \{4,8\}. \\
\end{cases}$$

We call a sum of weighted $m$-gonal numbers 
\begin{equation}F_m(\mathbf x)=a_1P_m(x_1)+\cdots+a_nP_m(x_n) \end{equation}
as {\it $m$-gonal form} where $a_i \in \N$.
For an $m$-gonal form $F_m(\mathbf x)$ and a positive integer $N \in \N$, if the diophantine equation 
\begin{equation}F_m(\mathbf x)=N \end{equation}
has an integer solution $\mathbf x \in \z^n$, then we say that $F_m(\mathbf x)$ {\it represents $N$}.
If $F_m(\mathbf x)$ represents every positive integer, then we say that $F_m(\mathbf x)$ is {\it universal}.
To determine the universality of $m$-gonal form is one of the significant problems in the field of number theory and has a long history.
Many mathematicians have been made an effort to classify universal $m$-gonal forms for some specific $m$'s.
Liouville classified all universal ternary triangular forms and Ramanujan classified all universal quaternary square forms.
Ramanujan claimed that there are 55 universal quateranry square forms, but one of them has been distinguised that is not universal.
Willerding claimed there are 124 quaternary non-diagonal universal quadratic form except Ramanujan's list of diagonal universal quadratic form.
But Willerding's list also has been corrected later as the total number of quaternary universal quadratic forms are exactly 214 up to isometry. 

In general, to determine the universality of a form is not easy.
On the other hand, in 1993, Conway and Schneeberger announced a stunning result on determining the universality of quadratic form.
Their result well known as {\it 15-Theorem} states that the representability of positive integers up to only $15$ characterize the universality of quadratic form.
The {\it $15$-Theorem} make readily check the universality of arbitrary quadratic form.
Actually, Willerding's mistake was corrected in virtue of the {\it $15$-Theorem}.
In \cite{BJ}, Kane and Liu claimed such a fineteness theorem holds for any $m$-gonal form for $m \ge 3$.
On the other words, they proved that for any $m \ge 3$, there exists (unique and minimal) $\gamma_m$ for which if an $m$-gonal form $F_m(\mathbf x)$ represents every positive integer up to $\gamma_m$, then $F_m(\mathbf x)$ is universal.
Since the smallest generalized $m$-gonal number is $m-3$ except $0$ and $1$, we may obviously have $m-4 \le \gamma_m$, i.e., $\gamma_m$ is asymptotically increasing value as $m$ increases.
They \cite{BJ} also questioned about the growth of $\gamma_m$ and showed that
$$m-4 \le \gamma_m \ll m^{7+\epsilon}.$$
As an improvement of the result, Kim and the author \cite{KP'} provided the optimal growth of $\gamma_m$ on $m$ which is exactly linear on $m$ by proving the following theorem.
\begin{thm} \label{c}
There exists an absolute constant $C>0$ such that $$m-4 \le \gamma_m \le C(m-2)$$
for any $m \ge 3$.
\end{thm}
\begin{proof}
See \cite{KP}.
\end{proof}
In effect, $\gamma_m$ is strictly bigger than $m-4$.
In other words, the representability of every positive integer up to $m-4$ by an $m$-gonal form is essential for its universality but does not characterize its universality.

Then how far is it to an $m$-gonal form which represents every positive integer up to $m-4$ from an universal form?
The following theorem (which is the main goal in this paper) may suggest the answer.

\begin{thm} \label{main thm}
\noindent $(1)$
For $m=6$ or $m \ge 12$, let an $m$-gonal form $F_m(\mathbf x)=a_1P_m(x_1)+\cdots+a_nP_m(x_n)$ represents every positive integer up to $m-4$.
Then there is $a_{n+1} \in \N$ for which
$$F_m(\mathbf x)+a_{n+1}P_m(x_{n+1})$$ is universal.

\noindent $(2)$ 
For $m \in \{5,7,8,9,10,11\}$, 
let an $m$-gonal form $F_m(\mathbf x)=a_1P_m(x_1)+\cdots+a_nP_m(x_n)$ represents every positive integer up to $m-4$.
Then there are $a_{n+1}, a_{n+2} \in \N$ for which
$$F_m(\mathbf x)+a_{n+1}P_m(x_{n+1})+a_{n+2}P_m(x_{n+2})$$ is universal.
\end{thm}

\vskip 0.5em

Theorem \ref{main thm} says that for any $m$-gonal form which is escalated up to $m-4$, only one more term would be enough to be universal for $m \ge 12$. 
Throughout this paper, we prove Theorem \ref{main thm} by using the arithmetic theory of quadratic form.
Especially, we concentrate the solvability of the diophantine systems
\begin{equation} \label{equ intro}
\begin{cases}
a_1x_1^2+\cdots+a_nx_n^2=2A+B \\
a_1x_1+\cdots+a_nx_n=B
\end{cases}
\end{equation}
over $\z$ to verify the representability of a positive integer by an $m$-gonal form.
Note that if (\ref{equ intro}) is solvable, then 
\begin{align*}
A(m-2)+B & =a_1P_m(x_1)+\cdots+a_nP_m(x_n) \\
&=\frac{(m-2)}{2}(\sum_{k=1}^na_kx_k^2-\sum_{k=1}^na_kx_k)+\sum_{k=1}^na_kx_k
\end{align*}
would have an integer solution, i.e., $A(m-2)+B$ is represented by the $m$-gonal form $a_1P_m(x_1)+\cdots+a_nP_m(x_n)$.
On the other hand, we may observe that the diophantine system (\ref{equ intro}) has an integer solution $(x_1,\cdots,x_n) \in \z^n$ if and only if the binary quadratic $\z$-lattice $[a_1+\cdots+a_n,B,2A+B]$ is represented by the diagonal quadratic $\z$-lattice $\left<a_1,a_2,\cdots,a_n\right>$ as follows
$$\begin{pmatrix}a_1+\cdots+a_n & B \\ B & 2A+B\end{pmatrix}=
\begin{pmatrix}1 & 1 & \cdots & 1 \\ x_1 & x_2 & \cdots & x_n\end{pmatrix}
\begin{pmatrix}a_1 & 0 & \cdots & 0 \\ 0 & a_2 & \cdots & 0 \\ \vdots & \vdots & \ddots & \vdots \\ 0 & 0 & \cdots & a_n \end{pmatrix}
\begin{pmatrix}1 & x_1 \\ 1 & x_2 \\ \vdots & \vdots \\1 & x_n \end{pmatrix}.$$
Throughout this paper, we use the above arguments.

\vskip 0.3cm
Before we move on, we organize some languages and notations which are adopted in this paper.
A {\it $\z$ (or $\z_p$)-lattice} is a $\z$ (or $\z_p$) module $L$ equipped with a quadratic form $Q$ on $L$ with $Q(L) \subseteq \z$ (or $\z_p$).
We often identify $L=\z x_1+\cdots+\z x_n$ with its {\it Gram matrix $\begin{pmatrix}(B(x_i,x_j)) \end{pmatrix}_{n\times n}$} where $B(x,y)=\frac{1}{2}(Q(x+y)-Q(x)-Q(y))$ is the symmetric bilinear form $B$ corresponding to $Q$.
When the Gram matrix $\begin{pmatrix}(B(x_i,x_j)) \end{pmatrix}_{n\times n}$ is diagonal, i.e., $B(x_i,x_j)=0$ for any $i\not=j$, we simply write $\begin{pmatrix}(B(x_i,x_j)) \end{pmatrix}_{n\times n}$ as $\left<Q(x_1),\cdots, Q(x_n)\right>$.
When $n=2$, we use the notation $[Q(x_1),B(x_1,x_2),Q(x_2)]$ instead of $\begin{pmatrix}(B(x_i,x_j)) \end{pmatrix}_{2\times 2}$.
On the other words, $[Q(\mathbf x_1),B(\mathbf x_1,\mathbf x_2),Q(\mathbf x_2)]$ identify with the binary quadratic $\z$-lattice $$Q(\mathbf x_1)X^2+2B(\mathbf x_1,\mathbf x_2)XY+Q(\mathbf x_2)Y^2.$$
In order to avoid confusing of the readers, I would like to note that someone may use the different notation $[Q(\mathbf x_1),2B(\mathbf x_1,\mathbf x_2),Q(\mathbf x_2)]$ for $Q(\mathbf x_1)X^2+2B(\mathbf x_1,\mathbf x_2)XY+Q(\mathbf x_2)Y^2$.
For simple notation, we write $a_1P_m(x_1)+\cdots+a_nP_m(x_n)$ as $\left<a_1,\cdots,a_n\right>_m$.
Any unexplained notation and terminology can be found in \cite{O1} and \cite{O}.
\vskip 0.3cm

\begin{section}*{Acknowledgements}
The author should like to express gratitude to Ben Kane for meticulous reading and many valuable comments which corrected wrong parts and typos and lead a lot of improvement. 
\end{section}

\vskip 0.3cm
\section{Lemmas}

\begin{prop} \label{l1}
Let an $m$-gonal form $F_m(\mathbf x)=a_1P_m(x_1)+\cdots+a_nP_m(x_n)$ with $a_1 \le \cdots \le a_n$ represents every positive integer up to $m-4$.
If an $m$-gonal form $G_m(\mathbf y)$ represents the consecutive $a_1+\cdots+a_i+1$ integers $N,\cdots, N+(a_1+\cdots+a_i)$, then the $m$-gonal form 
$$G_m(\mathbf y)+a_{i+1}P_m(x_{i+1})+\cdots+aP_n(x_n)$$ represents the consecutive $m-3$ integers $N,\cdots, N+(m-4)$.
\end{prop}
\begin{proof}
From assumption, for $0\le r \le m-4$, we may take $\mathbf x(r) \in \{0,1\}^n$
for which $$F_m(\mathbf x(r))=r.$$
And then for $\mathbf y(r)$ with $$G_m(\mathbf y(r))=N+a_1P_m(x_1(r))+\cdots +a_iP_m(x_i(r)),$$
we have that 
$$G_m(\mathbf y(r))+a_{i+1}P_m(x_{i+1}(r))+\cdots+aP_n(x_n(r))=N+r.$$
\end{proof}

\vskip 0.3em

\begin{coro} \label{rmk}
For an $m$-gonal form $$F_m(\mathbf x)=a_1P_m(x_1)+\cdots+a_nP_m(x_n)$$ with $a_1 \le \cdots \le a_n$ which represents every positive integer up to $m-4$ and an unary $m$-gonal form $$a_{n+1}P_m(x_{n+1}),$$ if 
$$a_1P_m(x_1)+\cdots +a_iP_m(x_i)+a_{n+1}P_m(x_{n+1})$$
(for some $i \le n$) represents every positive integer of the form of 
$$A(m-2)+B, \ A(m-2)+(B+1), \ \cdots, \ A(m-2)+(B+a_1+\cdots +a_i+1)$$
for a fixed $B \in \z$, then the $m$-gonal form 
$$F_m(\mathbf x)+a_{n+1}P_m(x_{n+1})$$ is universal.
\end{coro}
\begin{proof}
By using Proposition \ref{l1} with 
$$G_m(\mathbf y)=a_1P_m(y_1)+\cdots +a_iP_m(y_i)+a_{n+1}P_m(y_{n+1})$$ 
and
$$N=A(m-2)+B+1,$$ 
we obtain that $$F_m(\mathbf x)+a_{n+1}P_m(x_{n+1})$$ represents every positive integer in 
$$[A(m-2)+(B+1),A(m-2)+B+(m-3)],$$ yielding that $F_m(\mathbf x)+a_{n+1}P_m(x_{n+1})$ represents every positive integer in $[A(m-2)+B,A(m-2)+B+(m-3)]$.
Since $\{B, B+1, \cdots, B+(m-3)\}$ is a complete system of residues modulo $m-2$, it completes the proof.
\end{proof}

\vskip 0.3em

\begin{lem} \label{ml}
\noindent $(1)$ The $m$-gonal form $$\left<1,1,1,1\right>_m$$ represents every positive integer $A(m-2)+B$ such that 
$$\text{$B \eq 1 \pmod 2$ and $8A+4B-B^2 \ge 0$.}$$

\noindent $(2)$ The $m$-gonal form 
$$\left<1,1,1,2\right>_m$$ 
represents every positive integer $A(m-2)+B$ such that 
$$\text{$(A,B) \not\eq (0,0) \pmod 5$ and $10A+5B-B^2 \ge 0$.}$$

\noindent $(3)$ The $m$-gonal form 
$$\left<1,1,2,3\right>_m$$ 
represents every positive integer $A(m-2)+B$ such that  
$$\text{$(A,B) \not\eq (0,0) \pmod 7$ and $14A+7B-B^2 \ge 0$.}$$

\noindent $(4)$ The $m$-gonal form 
$$\left<1,1,2,4\right>_m$$ 
represents every positive integer $A(m-2)+B$ with $16A+8B-B^2 \ge 0$ otherwise 
$$A\eq 1 \pmod2\text{ and }B \eq 4 \pmod8$$
or
$$A\eq 0 \pmod2\text{ and }B \eq 0 \pmod8.$$

\noindent $(5)$ The $m$-gonal form 
$$\left<1,2,2,3\right>_m$$ 
represents every positive integer $A(m-2)+B$ such that  
$$\text{$B \not\eq 0 \pmod 4$ and $16A+8B-B^2 \ge 0$.}$$
\end{lem}
\begin{proof}
\noindent $(1)$ It would be enough to show that there is an integer solution $(x_1,x_2,x_3,x_4) \in \z^4$ for 
\begin{equation}\label{1111}\begin{cases}x_1+x_2+x_3+x_4=B \\ x_1^2+x_2^2+x_3^2+x_4^2=2A+B \end{cases}\end{equation}
where $\text{$B \eq 1 \pmod 2$ and $8A+4B-B^2 \ge 0$.}$
Recall that the diophantine system (\ref{1111}) has an integer solution $(x_1,x_2,x_3,x_4) \in \z^4$ if and only if 
the binary $\z$-lattice 
$[4,B,2A+B]$
is represented by the diagonal quaternary $\z$-lattice $\left<1,1,1,1\right>$
as follows 
$$\begin{pmatrix}4 & B \\ B & 2A+B\end{pmatrix}=
\begin{pmatrix}1 & 1 &1 & 1 \\ x_1 & x_2 & x_3 & x_4\end{pmatrix}
\begin{pmatrix}1 & 0 & 0 & 0 \\ 0 & 1 & 0 & 0 \\ 0 & 0 & 1 & 0 \\ 0 & 0 & 0 & 1 \end{pmatrix}
\begin{pmatrix}1 & x_1 \\ 1 & x_2 \\ 1 & x_3 \\1 & x_4 \end{pmatrix}.$$
By easy numeration, we may see that there are exactly two different ways 
$$4=\begin{pmatrix}1 & 1 &1 & 1  \\ \end{pmatrix}
\begin{pmatrix}1 & 0 & 0 & 0 \\ 0 & 1 & 0 & 0 \\ 0 & 0 & 1 & 0 \\ 0 & 0 & 0 & 1 \end{pmatrix}
\begin{pmatrix} 1  \\ 1  \\ 1  \\ 1  \end{pmatrix}$$
and
$$4=\begin{pmatrix}2 & 0 &0 & 0  \\ \end{pmatrix}
\begin{pmatrix}1 & 0 & 0 & 0 \\ 0 & 1 & 0 & 0 \\ 0 & 0 & 1 & 0 \\ 0 & 0 & 0 & 1 \end{pmatrix}
\begin{pmatrix} 2  \\ 0  \\ 0  \\ 0  \end{pmatrix}$$
to represent $4$ by $\left<1,1,1,1\right>$ up to isometry.
So we have that a binary quadratic $\z$-lattice $[4,B,2A+B]$ is represented by $\left<1,1,1,1\right>$ if and only if there is an integer solution $(x_1,x_2,x_3,x_4) \in \z^4$ for either 
$$\begin{pmatrix}4 & B \\ B & 2A+B\end{pmatrix}=
\begin{pmatrix}1 & 1 &1 & 1 \\ x_1 & x_2 & x_3 & x_4\end{pmatrix}
\begin{pmatrix}1 & 0 & 0 & 0 \\ 0 & 1 & 0 & 0 \\ 0 & 0 & 1 & 0 \\ 0 & 0 & 0 & 1 \end{pmatrix}
\begin{pmatrix}1 & x_1 \\ 1 & x_2 \\ 1 & x_3 \\1 & x_4 \end{pmatrix}$$
or
$$\begin{pmatrix}4 & B \\ B & 2A+B\end{pmatrix}=
\begin{pmatrix}2 & 0 &0 & 0 \\ x_1 & x_2 & x_3 & x_4\end{pmatrix}
\begin{pmatrix}1 & 0 & 0 & 0 \\ 0 & 1 & 0 & 0 \\ 0 & 0 & 1 & 0 \\ 0 & 0 & 0 & 1 \end{pmatrix}
\begin{pmatrix}2 & x_1 \\ 0 & x_2 \\ 0 & x_3 \\0 & x_4 \end{pmatrix},$$
i.e., there is an integer solution $(x_1,x_2,x_3,x_4) \in \z^4$ for either (\ref{1111})
or 
\begin{equation} \label{1111'}
\begin{cases}2x_1=B \\ x_1^2+x_2^2+x_3^2+x_4^2=2A+B.
\end{cases}
\end{equation}
From our assumption $B \eq 1 \pmod 2$, 
there would not be an integer solution $(x_1,x_2,x_3,x_4)$ for (\ref{1111'}).
So we have that the diophantine system (\ref{1111}) has an integer solution $(x_1,x_2,x_3,x_4) \in \z^4$ if and only if the binary $\z$-lattice 
$[4,B,2A+B]$
with $B\eq 1 \pmod2$ is represented by the diagonal quaternary $\z$-lattice $\left<1,1,1,1\right>$.

On the other hand, since the class number of $\left<1,1,1,1\right>$ is one, 
the local representability of 
$[4,B,2A+B]$
by $\left<1,1,1,1\right>$ would be equivalent with
the (global) representability of 
$[4,B,2A+B]$
by $\left<1,1,1,1\right>$.

When $p$ is an odd prime, since $\left<1,1,1,1\right> \cong \left<1,-1\right> \perp \left<1,-1\right>$ over $\z_p$ and $\left<1,-1\right> \otimes \z_p$ is  $\z_p$-universal,
we obtain that 
the binary $\z_p$-lattice $[4,B,2A+B] \otimes \z_p (\cong \left<4 , 4(8A+4B-B^2)\right> \otimes \z_p )$
is represented by $\left<1,1,1,1\right>  \otimes \z_p $.
When $p=2$, since $[4,B,2A+B] \otimes \z_2 $ (with $B \eq 1 \pmod 2$) is a binary odd unimodular $\z_2$-lattice with discriminant $8A+4B-B^2 \equiv 3 \pmod{8}$, $[4,B,2A+B] \otimes \z_2$ would be equivalent with $\left<1,3\right>\otimes \z_2$.
Since $\left<1,3\right>\otimes \z_2$ is clearly represented by $\left<1,1,1,1\right>\otimes \z_2$, we have that the positive definite binary $\z$-lattice $[4,B,2A+B]$ with $8A+4B-B^2 \ge 0$ is locally represented by $\left<1,1,1,1\right>$, yielding that $[4,B,2A+B]$ is indeed (globally) represented by $\left<1,1,1,1 \right>$.
This completes the proof.

(2) It would be enough to show that there is an integer solution $(x_1,x_2,x_3,x_4) \in \z^4$ for 
\begin{equation}\label{1112}\begin{cases}x_1+x_2+x_3+2x_4=B \\ x_1^2+x_2^2+x_3^2+2x_4^2=2A+B \end{cases}\end{equation}
where $\text{$(A,B) \not\eq (0,0) \pmod 5$ and $10A+5B-B^2 \ge 0$.}$
By easy numeration, we may see that there are exactly two different ways to represent $5$ by $\left<1,1,1,2\right>$ 
$$5=1\cdot 1^2+1\cdot 1^2+1\cdot 1^2+2\cdot 1^2$$
and
$$5=1\cdot 2^2+1\cdot 1^2+1\cdot 0^2+2\cdot 0^2$$
up to isometry.
Through similar arguments with the above (1), we may obtain that the binary quadratic $\z$-lattice $[5,B,2A+B]$ is represented by the diagonal quaternary $\z$-lattice $\left<1,1,1,2\right>$ if and only if there is an integer solution $(x_1,x_2,x_3,x_4) \in \z^4$ for either (\ref{1112}) or 
\begin{equation}\label{1112''}\begin{cases}2x_1+x_2=B \\ x_1^2+x_2^2+x_3^2+2x_4^2=2A+B.\end{cases}\end{equation}
For $(x_1,x_2,x_3,x_4) \in \z^4$ satisfying (\ref{1112''}), 
$$\left(\frac{x_1+x_3+2x_4}{2},x_2 ,\frac{x_1+x_3-2x_4}{2},\frac{x_1-x_3}{2} \right) \in \z^4$$
would be an integer solution for (\ref{1112}) since $2x_1+x_2 \eq x_1^2+x_2^2+x_3^2+2x_4^2 \pmod 2$.
So we conclude that the (\ref{1112}) has an integer solution $(x_1,x_2,x_3,x_4) \in \z^4$ if and only if the binary quadratic $\z$-lattice $[5,B,2A+B]$ is represented by the diagonal quaternary quadratic $\z$-lattice $\left<1,1,1,2\right>$

On the other hand, since the class number of $\left<1,1,1,2\right>$ is one, the local representability of $[5,B,2A+B]$ by $\left<1,1,1,2\right>$ would be equivalent with
the (global) representability of $[5,B,2A+B]$ by $\left<1,1,1,2\right>$.
When $p \not =5$ is an odd prime, since $\left<1,1,1,2\right> \cong \left<1,-2\right> \perp \left<1,-1\right>$ over $\z_p$ and $\left<1,-1\right> \otimes \z_p$ is $\z_p$-universal,
we obtain that the binary quadratic $\z_p$-lattice $[5,B,2A+B] \otimes \z_p (\cong \left<5 , 5(10A+5B-B^2)\right> \otimes \z_p )$ is represented by $\left<1,1,1,2\right>  \otimes \z_p $ since $5$ and $5(10A+5B-B^2)$ are represented by $\left<1,-2\right>\otimes \z_p$ and $\left<1,-1\right>\otimes \z_p$, respectively.
When $p=5$, $[5,B,2A+B] \otimes \z_5$ with $(A, B) \not \eq (0,0) \pmod 5$ is equivalent with the binary quadratic $\z_5$-lattice
$\left<2A+B(\in \z_5^{\times}), 2A+B(10A+5B-B^2)\right>\otimes \z_5$ (resp. $\left<1,1\right>\otimes \z_5$) when $B \equiv 0 \pmod 5$ (resp. when $B \not \equiv 0 \pmod5$).
So we may obtain that $[5,B,2A+B] \otimes \z_5$ with $(A, B) \not \eq (0,0) \pmod 5$ is represented by $\left<1,1,1,2\right> \otimes \z_5$.
When $p=2$, $\left<1,1,1,2\right>\otimes \z_2$ is equivalent with $\left(\left<5\right>\perp [2,1,2] \perp \left<2\cdot 7\right>\right)\otimes \z_2$.
So $[5,B,2A+B] \otimes \z_2 (\cong \left<5 , 5(10A+5B-B^2)\right> \otimes \z_2 )$ is represented by $\left<1,1,1,2\right>  \otimes \z_2 $ because 
$5(10A+5B-B^2)$ is even and
$\left( [2,1,2] \perp \left<2\cdot 7\right>\right)\otimes \z_2$ 
is even universal.
Over all, we obtain that the positive definite binary $\z$-lattice $[5,B,2A+B]$ with $10A+5B-B^2 \ge 0$ is locally represented by $\left<1,1,1,2\right>$, yielding that $[5,B,2A+B]$ is indeed (globally) represented by $\left<1,1,1,2 \right>$.
This completes the proof.

(3) It would be enough to show that there is an integer solution $(x_1,x_2,x_3,x_4) \in \z^4$ for 
\begin{equation}\label{1123}\begin{cases}x_1+x_2+2x_3+3x_4=B \\ x_1^2+x_2^2+2x_3^2+3x_4^2=2A+B \end{cases}\end{equation}
where $\text{$(A,B) \not\eq (0,0) \pmod 7$ and $14A+7B-B^2 \ge 0$.}$
Since there are exactly three different ways to represent $7$ by $\left<1,1,2,3\right>$
$$7=1\cdot 1^2+1\cdot 1^2+2\cdot 1^2+3\cdot 1^2,$$
$$7=1\cdot 2^2+1\cdot 1^2+2\cdot 1^2+3\cdot 0^2,$$
and
$$7=1\cdot 2^2+1\cdot 0^2+2\cdot0^2+3\cdot 1^2$$
up to isomery, through the similar arguments with the above (1), we may obtain that 
the binary quadratic $\z$-lattice $[7,B,2A+B]$ is represented by the diagonal quaternary $\z$-lattice $\left<1,1,2,3\right>$ if and only if
there is $(x_1,x_2,x_3,x_4) \in \z^4$ satisfying either (\ref{1123}) or 
\begin{equation}\label{1123''}\begin{cases}2x_1+x_2+2x_3=B \\ x_1^2+x_2^2+2x_3^2+3x_4^2=2A+B\end{cases}\end{equation}
or
\begin{equation}\label{1123'''}\begin{cases}2x_1+3x_4=B \\ x_1^2+x_2^2+2x_3^2+3x_4^2=2A+B.\end{cases}\end{equation}
For $(x_1,x_2,x_3,x_4) \in \z^4$ satisfying (\ref{1123''}),
$$\left(\frac{x_1+3x_4}{2},x_2 ,x_3,\frac{x_1-x_4}{2} \right) \in \z^4 $$
would be an integer solution for (\ref{1123}) since $x_1+x_4 \eq 0 \pmod2$.
For $(x_1,x_2,x_3,x_4) \in \z^4$ satisfying (\ref{1123'''}),
$$\left(\frac{x_1+x_2+2x_3}{2},\frac{x_1+x_2-2x_3}{2} ,\frac{x_1-x_2}{2},x_4 \right) \in \z^4 $$
would be an integer solution for (\ref{1123}) since $x_1+x_2 \eq 0 \pmod2$.
So we conclude that the (\ref{1123}) has an integer solution $(x_1,x_2,x_3,x_4) \in \z^4$ if and only if the binary quadratic $\z$-lattice $[7,B,2A+B]$ is represented by the diagonal quaternary quadratic $\z$-lattice $\left<1,1,2,3\right>$.

On the other hand, since the class number of $\left<1,1,2,3\right>$ is one, the local representability of $[7,B,2A+B]$ by $\left<1,1,2,3\right>$ would be equivalent with the (global) representability of $[7,B,2A+B]$ by $\left<1,1,2,3\right>$.
When $p \not =7$ is an odd prime, since $\left<1,1,2,3\right> \cong \left<1,-6\right> \perp \left<1,-1\right>$ over $\z_p$ and $\left<1,-1\right> \otimes \z_p$ is $\z_p$-universal,
we obtain that $[7,B,2A+B] \otimes \z_p (\cong \left<7 , 7(14A+7B-B^2)\right> \otimes \z_p )$ is represented by $\left<1,1,2,3\right>  \otimes \z_p $ since $7$ and $7(14A+7B-B^2)$ are represented by $\left<1,-6\right>\otimes \z_p$ and $\left<1,-1\right>\otimes \z_p$, respectively.
When $p=7$, $[7,B,2A+B] \otimes \z_7$ with $(A, B) \not \eq (0,0) \pmod 7$ is equivalent with 
$\left<2A+B(\in \z_7^{\times}), 2A+B(14A+7B-B^2)\right>\otimes \z_7$ (resp. $\left<1,6\right>\otimes \z_7$) when $B \equiv 0 \pmod 7$ (resp. when $B \not \equiv 0 \pmod 7$).
So we may obtain that $[7,B,2A+B] \otimes \z_7$ is represented by $\left<1,1,2,3\right> \otimes \z_7$.
When $p=2$, we may see that $[7,B,2A+B] \otimes \z_2 (\cong \left<7 , 7(14A+7B-B^2)\right> \otimes \z_2 )$ is represented by $\left<1,1,2,3\right>  \otimes \z_2 \cong ( \left<7\right>\perp [0,1,0] \perp$ $\left<2\cdot 3\right> )\otimes \z_2$ since 
$7(14A+7B-B^2)$ is even and
$  [0,1,0] \otimes \z_2$ 
is even universal.
Over all, we have that the positive definite binary $\z$-lattice $[7,B,2A+B]$ with $14A+7B-B^2 \ge 0$ is locally represented by $\left<1,1,2,3\right>$, yielding that $[7,B,2A+B]$ is indeed (globally) represented by $\left<1,1,2,3 \right>$.
This completes the proof.

(4) It would be enough to show that there is an integer solution $(x_1,x_2,x_3,x_4) \in \z^4$ for 
\begin{equation}\label{1124}\begin{cases}x_1+x_2+2x_3+4x_4=B \\ x_1^2+x_2^2+2x_3^2+4x_4^2=2A+B \end{cases}\end{equation}
for $A$ and $B$ satisfying the assumption.
Since there are exactly four different ways to represent $8$ by $\left<1,1,2,4\right>$
$$8=1\cdot 1^2+1\cdot 1^2+2\cdot 1^2+4\cdot 1^2,$$
$$8=1\cdot 2^2+1\cdot 0^2+2\cdot 0^2+4\cdot 1^2,$$
$$8=1\cdot 2^2+1\cdot 2^2+2\cdot 0^2+4\cdot 0^2,$$
and
$$8=1\cdot 0^2+1\cdot 0^2+2\cdot2^2+3\cdot 0^2$$
up to isomery, through the similar arguments with the above (1), we may obtain that 
the binary quadratic $\z$-lattice $[8,B,2A+B]$ is represented by the diagonal quaternary $\z$-lattice $\left<1,1,2,4\right>$ if and only if
there is $(x_1,x_2,x_3,x_4) \in \z^4$ satisfying either (\ref{1124}) or 
\begin{equation}\label{1124''}\begin{cases}2x_1+4x_4=B \\ x_1^2+x_2^2+2x_3^2+4x_4^2=2A+B,\end{cases}\end{equation}
\begin{equation}\label{1124'''}\begin{cases}2x_1+2x_2=B \\ x_1^2+x_2^2+2x_3^2+4x_4^2=2A+B,\end{cases}\end{equation}
or
\begin{equation}\label{1124''''}\begin{cases}4x_3=B \\ x_1^2+x_2^2+2x_3^2+4x_4^2=2A+B.\end{cases}\end{equation}
When $B \eq 1 \pmod 2$, there is no integer solution for (\ref{1124''}), (\ref{1124'''}), and (\ref{1124''''}).
When $B \eq 2 \pmod 4$, we may not have an integer solution for (\ref{1124'''}) and (\ref{1124''''}).
For an integer solution $(x_1,x_2,x_3,x_4) \in \z^4$ for (\ref{1124''}), $$\left(\frac{x_1+x_2+2x_3}{2},\frac{x_1+x_2-2x_3}{2},\frac{x_1+x_2}{2},x_4\right) \in \z^4$$ would be an integer solution for (\ref{1124}). 
When $B \eq 0 \pmod 4$, for an integer solution $(x_1,x_2,x_3,x_4) \in \z^4$ for (\ref{1124'''}), $$\left(\frac{x_1-x_2+2x_3}{2},\frac{x_1-x_2-2x_3}{2},\frac{x_1+x_2}{2},x_4\right) \in \z^4$$ would be an integer solution for (\ref{1124''''}), so we may assume that there is an integer solution $(x_1,x_2,x_3,x_4) \in \z^4$ satisfying  (\ref{1124}) or (\ref{1124''''}).
For an integer solution $(x_1,x_2,x_3,x_4) \in \z^4$ for (\ref{1124''''}) with 
$$A \eq 0 \pmod 2, B \eq 4 \pmod8$$
or
$$A \eq 1 \pmod 2, B \eq 0 \pmod8,$$
$x_1\eq x_2 \eq 1 \pmod 2$ holds.
By changing $x_1$ to $-x_1$ if it is necessary, we may assume that $-x_1+x_2+2x_3 \eq 0 \pmod 4$.
And then for an integer solution $(x_1,x_2,x_3,x_4) \in \z^4$ for (\ref{1124''''}) with $-x_1+x_2+2x_3 \eq 0 \pmod 4$, $$\left(\frac{-x_1-3x_2+2x_3+4x_4}{4},\frac{3x_1+x_2+2x_3+4x_4}{4}, \frac{x_1-x_2+2x_3-4x_4}{4},\frac{-x_1+x_2+2x_3}{4} \right) \in \z^4$$ would be an integer soultion for (\ref{1124}).
So we may conclude that the (\ref{1124}) has an integer solution $(x_1,x_2,x_3,x_4) \in \z^4$ if and only if the binary quadratic $\z$-lattice $[8,B,2A+B]$ is represented by the diagonal quaternary quadratic $\z$-lattice $\left<1,1,2,4\right>$.

On the other hand, since the class number of $\left<1,1,2,4\right>$ is one, the local representability of $[8,B,2A+B]$ by $\left<1,1,2,4\right>$ would be equivalent with
the (global) representability of  $[8,B,2A+B]$ by $\left<1,1,2,4\right>$.
When $p $ is an odd prime, since $\left<1,1,2,4\right> \cong \left<1,-2\right> \perp \left<1,-1\right>$ over $\z_p$ and $\left<1,-1\right> \otimes \z_p$ is $\z_p$-universal,
we obtain that $[8,B,2A+B] \otimes \z_p (\cong \left<8 , 8(16A+8B-B^2)\right> \otimes \z_p )$ is represented by $\left<1,1,2,4\right>  \otimes \z_p $ since $8$ and $8(16A+8B-B^2)$ are represented by $\left<1,-2\right>\otimes \z_p$ and $\left<1,-1\right>\otimes \z_p$, respectively.
When $p=2$, $[8,B,2A+B] \otimes \z_2$ is equivalent with 
$$\text{$\left<1,7\right>$ (when $B \eq 1 \pmod 2$),}$$ 
$$\text{$\left<2,6\right>, \left<2,14\right>,$ or $[0,2,0]$ (when $B \eq 2 \pmod 4$),}$$
$$\text{$\left<2A+B, \frac{16A+8B-B^2}{2A+B} \right>$ (when $B \eq 0 \pmod 4$ and $2A+B \not\eq 0 \pmod 8$)},$$
$$\text{$[0,4,0]$ or $[8,4,8]$ (when $B \eq 0 \pmod 4$ and $2A+B \eq 0 \pmod 8$)}.$$
One may check that the above binary $\z_2$-lattices are represented by the diagonal quaternary $\z_2$-lattice $\left<1,1,2,4\right>\otimes \z_2$ case by case.
Consequently, we may obtain that the positive definite binary $\z$-lattice $[8,B,2A+B]$ with $16A+8B-B^2 \ge 0$ is locally represented by $\left<1,1,2,4\right>$, yielding that $[8,B,2A+B]$ is indeed (globally) represented by $\left<1,1,2,4 \right>$.
This completes the proof.

(5) It would be enough to show that there is an integer solution $(x_1,x_2,x_3,x_4) \in \z^4$ for 
\begin{equation}\label{1223}\begin{cases}x_1+2x_2+2x_3+3x_4=B \\ x_1^2+2x_2^2+2x_3^2+3x_4^2=2A+B \end{cases}\end{equation}
where $\text{$B \not\eq 0 \pmod 4$ and $16A+8B-B^2 \ge 0$.}$
By easy numeration, we may see that there are exactly two different ways to represent $8$ by $\left<1,2,2,3\right>$ up to isometry as follows,
$$8=1\cdot1^2+2\cdot1^2+ 2\cdot1^2+3\cdot1^2$$
and
$$8=1\cdot2^2+2\cdot1^2+ 2\cdot1^2+3\cdot0^2.$$
Through the similar arguments with the above, we may see that
the binary quadratic $\z$-lattice $[8,B,2A+B]$ is represented by the diagonal quaternary quadratic $\z$-lattice $\left<1,2,2,3\right>$
if and only if there is an integer solution $(x_1,x_2,x_3,x_4) \in \z^4$ satisfying either (\ref{1223}) or 
\begin{equation}\label{1223''}\begin{cases}2x_1+2x_2+2x_3=B \\ x_1^2+2x_2^2+2x_3^2+3x_4^2=2A+B.\end{cases}\end{equation}
For an integer soultion $(x_1,x_2,x_3,x_4) \in \z^4$ for (\ref{1223''}), $$\left(\frac{x_1+3x_4}{2},x_2,x_3,\frac{x_1-x_4}{2}\right) \in \z^4$$ would be an integer solution for (\ref{1223}).
So we conclude that the diophantine system (\ref{1223}) has an integer solution $(x_1,x_2,x_3,x_4) \in \z^4$ if and only if the binary quadratic $\z$-lattice $[8,B,2A+B]$ is represented by the diagonal quaternary quadratic $\z$-lattice $\left<1,1,2,4\right>$.

On the other hand, since the class number of $\left<1,2,2,3\right>$ is one, the local representability of $[8,B,2A+B]$ by $\left<1,2,2,3\right>$ would be equivalent with
the (global) representability of  $[8,B,2A+B]$ by $\left<1,2,2,3\right>$.
When $p$ is an odd prime, one may easily see that $[8,B,2A+B]\otimes \z_p \cong \left<8,8(16A+8B-B^2)\right>\otimes \z_p$ is represented by $\left<1,2,2,3\right> \otimes \z_p$.
When $p=2$, $[8,B,2A+B] \otimes \z_2$ is equivalent with 
$$\text{$\left<1,7\right>$ when $B \eq 1 \pmod 2$,}$$ 
or
$$\text{$\left<2,6\right>, \left<2,14\right>,$ or $[0,2,0]$ when $B \eq 2 \pmod 4$}.$$
One may check that the above binary $\z_2$-lattices are represented by the diagonal quaternary $\z_2$-lattice $\left<1,2,2,3\right>\otimes \z_2$ case by case.
Consequently, we have that the positive definite binary $\z$-lattice $[8,B,2A+B]$ with $16A+8B-B^2 \ge 0$ is locally represented by $\left<1,2,2,3\right>$, yielding that $[8,B,2A+B]$ is indeed (globally) represented by $\left<1,2,2,3 \right>$.
This completes the proof.
\end{proof}

\vskip 0.5em

\section{Main Theorem}

Suppose that an $m$-gonal form $$F_m(\mathbf x)=a_1P_m(x_1)+\cdots+a_nP_m(x_n)$$ represents every positive integer up to $m-4$.
Without loss of generality, we assume that $a_1\le \cdots \le a_n$.
Since the smallest $m$-gonal number is $m-3$ except $0$ and $1$, we conclude that 
\begin{equation} \label{b}
\begin{cases}
a_1=1 & \\
a_{i+1} \le a_1+a_2+\cdots+a_i+1 & \text{if } a_1+a_2+\cdots+a_i<m-4\\
a_1+\cdots+a_n \ge m-4.
\end{cases}
\end{equation}
Following (\ref{b}), for $m \ge 12$, we have that $n \ge 4$ and
\begin{equation} \label{3}
(a_1,a_2,a_3) \in \{(1,1,1),(1,1,2),(1,1,3),(1,2,2),(1,2,3),(1,2,4)\}.
\end{equation}
In Lemma \ref{mml}, for each $(a_1,a_2,a_3)$ in (\ref{3}), we assign $a\in \N$ for which for any $m$-gonal form $F_m(\mathbf x)$ which has its first three coefficients as $(a_1,a_2,a_3)$ and represents every positive integer up to $m-4$, $$F_m(\mathbf x)+aP_m(x_{n+1})$$ is universal to yield Theorem \ref{main thm} (1) for $m \ge 12$.

\vskip 1em

\begin{lem} \label{mml}
Let an $m$-gonal form $F_m(\mathbf x)=a_1P_m(x_1)+\cdots +a_nP_m(x_n)$ with $a_1 \le \cdots \le a_n$ represents every positive integer up to $m-4$.

\noindent $(1)$ When $(a_1,a_2,a_3)=(1,1,1)$, $$F_m(\mathbf x)+2P_m(x_{n+1})$$  with $m \ge 8$ is universal.

\noindent $(2)$ When $(a_1,a_2,a_3)=(1,1,2)$, $$F_m(\mathbf x)+3P_m(x_{n+1})$$ is universal.

\noindent $(3)$ When $(a_1,a_2,a_3)=(1,1,3)$, $$F_m(\mathbf x)+2P_m(x_{n+1})$$ with $m \ge 10$ is universal.

\noindent $(4)$ When $(a_1,a_2,a_3)=(1,2,2)$ and $a_4 \not=4$ (resp. $a_4=4$), $$ \text{$F_m(\mathbf x)+3P_m(x_{n+1})$ (resp. $F_m(\mathbf x)+6P_m(x_{n+1})$)}$$ with $m \ge 10$ is universal.


\noindent $(5)$ When $(a_1,a_2,a_3)=(1,2,3)$ and $a_4 \not=7$ (resp. $a_4=7$), $$\text{$F_m(\mathbf x)+P_m(x_{n+1})$ (resp. $F_m(\mathbf x)+2P_m(x_{n+1})$)}$$ with $m \ge 11$ is universal.

\noindent $(6)$ When $(a_1,a_2,a_3)=(1,2,4)$ and $a_4 \not=8$ (resp. $a_4=8$), $$\text{$F_m(\mathbf x)+P_m(x_{n+1})$ (resp. $F_m(\mathbf x)+2P_m(x_{n+1})$)}$$ with $m \ge 12$ is universal.
\end{lem}
\begin{proof}
(1) By Lemma \ref{ml} (2), $$a_1P_m(x_1)+a_2P_m(x_2)+a_3P_m(x_3)+2P_m(x_{n+1})$$ represents every positive integer of the form of 
$$A(m-2)+1, \ A(m-2)+2, \ A(m-2)+3, \ A(m-2)+4.$$
For $m \ge8$, since $1 \le a_4 \le 4$, we may see that $$a_1P_m(x_1)+a_2P_m(x_2)+a_3P_m(x_3)+a_4P_m(x_4)+2P_m(x_{n+1})$$ represents every positive integer of the form of 
$$A(m-2)+1, \ A(m-2)+2, \  \cdots , \ A(m-2)+(4+a_4)$$ with $x_4 \in \{0,1\}$.
On the other hand, by using Lemma \ref{ml} (2) again, we may see that for each $1 \le a_4 \le 4$, 
$$a_1P_m(x_1)+a_2P_m(x_2)+a_3P_m(x_3)+2P_m(x_{n+1})$$ represents every positive integer of the form of 
$$\begin{cases}A(m-2)-a_4 & \text{ for } A > a_4\\
A(m-2) & \text{ for } A \le a_4(<5)
 \end{cases}$$
which yields that 
$$a_1P_m(x_1)+a_2P_m(x_2)+a_3P_m(x_3)+a_4P_m(x_4)+2P_m(x_{n+1})$$ represents every positive integer of the form of $$A(m-2)$$
with $x_4 \in \{0,1\}$.
Over all, we obtain  that
$$a_1P_m(x_1)+a_2P_m(x_2)+a_3P_m(x_3)+a_4P_m(x_4)+2P_m(x_{n+1})$$
represents every positive integer of the form of 
$$A(m-2), \ A(m-2)+1, \cdots , \ A(m-2)+(a_1+a_2+a_3+a_4+1).$$
It completes the proof by Corollary \ref{rmk}.

(2) By Lemma \ref{ml} (3), $$a_1P_m(x_1)+a_2P_m(x_2)+a_3P_m(x_3)+3P_m(x_{n+1})$$ represents every positive integer of the form of 
$$A(m-2)+1, \ A(m-2)+2, \cdots, \ A(m-2)+6.$$
It completes the proof by Corollary \ref{rmk}.

(3) One may show this similarly with the above by using Lemma \ref{ml} (3).

(4) One may show this similarly with above by using Lemma \ref{ml} (5) when $a_4\not=4$.

Now consider $(a_1,a_2,a_3,a_4)=(1,2,2,4)$. 
By Lemma \ref{ml} (3),
$$a_2P_m(x_2)+a_3P_m(x_3)+a_4P_m(x_4)+6P_m(x_{n+1})$$
would represent every positive integer of the form of
$$2A(m-2)+2,\ 2A(m-2)+4, \ \cdots, \ 2A(m-2)+12$$
where $A \in \N_0$.
So we may see that 
$$a_1P_m(x_1)+a_2P_m(x_2)+a_3P_m(x_3)+a_4P_m(x_4)+6P_m(x_{n+1})$$
represents the positive integers 
$$2A(m-2)+3,\ 2A(m-2)+5, \ \cdots, \ 2A(m-2)+13$$
with $x_1=1$,
$$(2A+3)(m-2)+2,\ (2A+1)(m-2)+4, \ \cdots, \ (2A+1)(m-2)+12$$ 
with $x_1 \in \{-2,2\}$,
and
$$(2A+1)(m-2)+1,\ (2A+1)(m-2)+3, \ \cdots, \ (2A+1)(m-2)+11$$
with $x_1 =-1$
where $A \in \N_0$.
Since $(m-2)+2=m$ is also clearly represented by $a_1P_m(x_1)+a_2P_m(x_2)+a_3P_m(x_3)+a_4P_m(x_4)+6P_m(x_{n+1})$, we may conclude that 
$$a_1P_m(x_1)+a_2P_m(x_2)+a_3P_m(x_3)+a_4P_m(x_4)+6P_m(x_{n+1})$$ represents every positive integer of the form of
$$A(m-2)+2,\ A(m-2)+3, \ \cdots, \ A(m-2)+12.$$
It completes the proof by Corollary \ref{rmk}.

(5) One may show this similarly with the above by using Lemma \ref{ml} (3) (resp. Lemma \ref{ml} (5)) when $a_4\not=7$ (resp. when $a_4=7$).

(6) One may show this similarly with the above by using Lemma \ref{ml} (4) when $a_4\not=8$ by showing that 
$$a_1P_m(x_1)+a_2P_m(x_2)+a_3P_m(x_3)+a_4P_m(x_4)+P_m(x_{n+1})$$
represents every positive integer of the form of 
$$A(m-2)-1,\ A(m-2),\ A(m-2)+1, \ \cdots, \ A(m-2)+(a_1+a_2+a_3+a_4).$$
For $(a_1,a_2,a_3,a_4)=(1,2,4,8)$, one may show that 
$$a_1P_m(x_1)+a_2P_m(x_2)+a_3P_m(x_3)+a_4P_m(x_4)+2P_m(x_{n+1})$$
represents every positive integer of the form of
$$A(m-2)-1,\ A(m-2),\ A(m-2)+1, \ \cdots, \ A(m-2)+(a_1+a_2+a_3+a_4)$$
through similar processing with the case $(a_1,a_2,a_3,a_4)=(1,2,2,4)$ in (4) by using Lemma \ref{ml} (4).
\end{proof}

\vskip 1em

\begin{lem} \label{mml'}
Let an $m$-gonal form $F_m(\mathbf x)=a_1P_m(x_1)+\cdots +a_nP_m(x_n)$ with $a_1 \le \cdots \le a_n$ represents every positive integer up to $m-4$.

\noindent $(1)$ When $(a_1,a_2)=(1,1)$, $$F_m(\mathbf x)+2P_m(x_{n+1})+3P_m(x_{n+2})$$  is universal.

\noindent $(2)$ When $(a_1,a_2)=(1,2)$, $$F_m(\mathbf x)+P_m(x_{n+1})+3P_m(x_{n+2})$$ is universal.
\end{lem}
\begin{proof}
By Lemma \ref{ml} (3), since $\left<1,1,2,3\right>_m$ represents every positive integer of the form of 
$$A(m-2)+1, \ A(m-2)+2, \cdots, \ A(m-2)+6,$$
we may obtain this by Proposition \ref{rmk}.
\end{proof}

\vskip 1em

\begin{rmk}
\noindent $(1)$ 
Theorem \ref{main thm} directly follows from Lemma \ref{mml} and Lemma \ref{mml'} for $m \not=5,6$.
For $m=5$ and $6$, since $\left<1,1,1\right>_5$ and $\left<1,1,2\right>_6$ are universal, we may complete the proof of Theorem \ref{main thm}.

\noindent $(2)$ In order for an $m$-gonal form $F_m(\mathbf x)=a_1P_m(x_1)+\cdots +a_nP_m(x_n)$ with $a_1 \le \cdots \le a_n$ to represent every positive integer up to $m-4$, the coefficients should satisfy (\ref{b}).
This yields that the minimal rank for an $m$-gonal form which represents every positive integer up to $m-4$ is $\ceil{\log_2(m-3)}$.
Of course, the minimal rank for universal $m$-gonal form would also be greater than or equal to $\ceil{\log_2(m-3)}$.
The author determined the minimal rank for universal $m$-gonal form for sufficiently large $m$ in her thesis \cite{P} by showing Theorem \ref{min uni rank}.
\begin{thm}\label{min uni rank}
For all sufficiently large $m > 2 \left((2C+\frac{1}{4})^{\frac{1}{4}}+\sqrt{2}\right)^2$, the minimal rank for universal $m$-gonal form is 
$$ \begin{cases}\ceil{\log_2(m-3)}+1& \text{ when } -3 \le 2^{\ceil{\log_2(m-3)}}-m\le 1\\
\ceil{\log_2(m-3)} & \text{ when } \ \quad 2 \le 2^{\ceil{\log_2(m-3)}}-m
\end{cases}$$
where $C$ is the constant in Theorem \ref{c}.
\end{thm}
\begin{proof}
See \cite{P}.
\end{proof}
But it is unknown the minimal rank for universal $m$-gonal form for the remaining $m \le 2 \left((2C+\frac{1}{4})^{\frac{1}{4}}+\sqrt{2}\right)^2$.
On the other hand, Theorem \ref{main thm} (1) may give a quite reasonable bound on the minimal rank for universal $m$-gonal form for $m\ge12$.
\end{rmk}

\vskip 1em

\begin{thm}
For $m\ge 12$, the minimal rank for universal $m$-gonal form would be
$$ \begin{cases}\ceil{\log_2(m-3)}+1& \text{ when } -3 \le 2^{\ceil{\log_2(m-3)}}-m\le 1\\
\ceil{\log_2(m-3)} \text{ or } \ceil{\log_2(m-3)}+1 & \text{ when } \ \quad 2 \le 2^{\ceil{\log_2(m-3)}}-m.
\end{cases} $$
\end{thm}
\begin{proof}
Note that the minimal rank for an $m$-gonal form which represents every positive integer up to $m-4$ is $\ceil{\log_2(m-3)}$.
So we have the minimal rank for universal $m$-gonal form is also greater than or equal to $\ceil{\log_2(m-3)}$
and by Theorem \ref{main thm}, the minimal rank for universal $m$-gonal form is less than or equal to $\ceil{\log_2(m-3)}+1$. 
When $-3 \le 2^{\ceil{\log_2(m-3)}}-m\le 1$, one may deduce that there is no universal $m$-gonal form of rank $\ceil{\log_2(m-3)}$ throughout similar processing with the proof of Theorem \ref{min uni rank}. 
This completes the proof.
\end{proof}

\end{document}